\documentclass[11pt]{amsart}
\usepackage{amssymb}
\usepackage{times}
\usepackage{url}
\usepackage{hyperref}
\usepackage{xcolor}

\usepackage{rotating}
\usepackage{pdflscape}

\newtheorem{theorem}{Theorem}[section]
\newtheorem{lemma}[theorem]{Lemma}

\newtheorem{corollary}[theorem]{Corollary}
\theoremstyle{remark}
\newtheorem{remark}[theorem]{Remark}

\theoremstyle{example}

\newcommand{\F}{\mathbb{F}}
\newcommand{\Z}{\mathbb{Z}}
\newcommand{\Q}{\mathbb{Q}}

\newcommand{\G}{{\Gamma_n}}

\newcommand{\GL}{\mathrm{GL}}

\newcommand{\SL}{\mathrm{SL}}
\newcommand{\PSL}{\mathrm{PSL}}
\newcommand{\Sp}{\mathrm{Sp}}

\definecolor{grey}{rgb}{0.6,0.6,0.6}
\def\ppqf#1#2{{\color{grey}\mathop{\underline{{\color{black}#1}}%
}\limits_{{\scriptscriptstyle #2}}}}

\begin{document}

\title[Algorithms for experimenting with Zariski dense 
subgroups]{Algorithms for experimenting with Zariski dense subgroups}

\author{A.~S.~Detinko}
\address{School of Computer Science\\
University of St~Andrews\\
North Haugh\\
St~Andrews KY16 9SX\\
UK}
\email{ad271@st-andrews.ac.uk}

\author{D.~L.~Flannery}
\address{
School of Mathematics, Statistics and Applied Mathematics\\
National University of Ireland Galway\\
University Road\\
Galway H91TK33\\
Ireland}
\email{dane.flannery@nuigalway.ie}

\author{A.~Hulpke}
\address{Department of Mathematics\\
Colorado State University\\
Fort Collins\\
CO 80523-1874\\
USA}
\email{Alexander.Hulpke@colostate.edu}

\footnotetext{To appear in Experimental Mathematics.}

\begin{abstract}
We give a method to describe all congruence images 
of a finitely generated Zariski dense group $H\leq \SL(n, \Z)$. 
The method is applied to obtain efficient algorithms for solving this
problem in odd prime degree $n$; if $n=2$ then we compute 
all congruence images only modulo primes. 
We propose a separate method that works for all $n$ 
if $H$ contains a known transvection.
The algorithms have been implemented in {\sf GAP}, enabling
computer experiments with important classes of linear groups 
that have recently emerged.
\end{abstract}

\maketitle

\section{Introduction}

This paper further develops methods and algorithms for computing with 
linear groups over infinite domains.
It is a sequel to \cite{Density}.

Let $H$ be a finitely generated subgroup of $\SL(n, \Z)$, $n\geq 2$, 
that is Zariski dense in $\SL(n, \mathbb{R})$. By the strong 
approximation theorem, $H$ is congruent to $\SL(n, p)$ modulo $p$ for all 
but a finite number of primes $p$~\cite[p.~391]{LubotzkySegal}. 
If $n > 2$ and $H$ is arithmetic, i.e., $H$ has finite index in $\SL(n, \Z)$, 
then the congruence subgroup property guarantees that $H$ contains a 
\emph{principal congruence subgroup of level $m$} for some $m$, i.e., 
the kernel $\G_{, m}$ of the reduction modulo $m$ homomorphism 
$\varphi_m: \SL(n, \Z)\rightarrow\SL(n, \Z_m)$. In that event $H$
contains a unique maximal principal congruence subgroup $\G_{,M}$, and
we call $M = M(H)$ the \emph{level of $H$}. 
The dense group $H$ is contained in a uniquely defined 
minimal arithmetic overgroup $\mathrm{cl}(H)$, namely the intersection 
of all arithmetic subgroups of $\SL(n, \Z)$ containing $H$ (its 
`arithmetic closure')~\cite[Section 3.3]{Density}. 
The level of $H$ is defined to be the level of 
$\mathrm{cl}(H)$, and is again denoted $M(H)$. 
Sarnak~\cite{Sarnak} calls dense non-arithmetic
 $H\leq \SL(n, \Z)$ a \emph{thin} matrix group. 

In \cite{Density} we developed practical algorithms to compute the level 
$M$ of a dense group $H \leq \SL(n, \Z)$ for $n > 2$. This was 
motivated by the fact that $M$ is the key component of our algorithms 
to compute with arithmetic subgroups of $\SL(n, \Z)$ \cite{Arithm}. 
Once we have $M$, we can find $\mathrm{cl}(H)$ and obtain 
further information about dense $H$ via computation with $\mathrm{cl}(H)$.

Algorithms to compute $M$ were implemented
and used to carry out extensive computer experiments, as detailed in 
\cite[Section 4]{Density}. Our method requires the set 
$\Pi = \Pi(H)$ of primes $p$ such that 
$\varphi_p(H) \neq \SL(n, p)$: essentially, a computational realization of
the strong approximation theorem. 

The aim of the present work is twofold. First, in
Section~\ref{SAandRecognitionOfImages} we establish a general 
method to compute $\Pi(H)$ for dense $H \leq \SL(n, \Z)$, 
drawing on the classification of maximal subgroups 
of $\SL(n, p)$ as in \cite{Aschbacher84}
(see also \cite[p.~397]{LubotzkySegal}). 
This is then applied in Section~\ref{AlgorithmsSA} to obtain efficient 
algorithms to compute $\Pi(H)$ for prime degree $n$ (in which case 
the types of maximal subgroups of $\SL(n, p)$ are quite restricted). 
Moreover, for odd prime $n$, we build on this knowledge to
describe the congruence images of $H$ modulo all positive integers.  
Arbitrary degrees $n$ are treated in \cite{SAT_General} 
(albeit with algorithms that are less efficient for prime $n$ than those herein).

We also give an algorithm to compute $\Pi(H)$ for subgroups $H$ 
of $\SL(2n, \Z)$ that contain a known 
transvection (a unipotent element $t$ such that $t - 1_n$ 
has matrix rank $1$). 
This completes the task begun in \cite[Section 3.2]{Density}.

Another goal is to perform computer experiments successfully 
with low-dimensional dense representations 
of finitely presented groups that have recently been the focus of 
much attention. We compute $\Pi$ and $M$ for each group, thus 
enabling us to describe
all of its congruence quotients. Experimental results are
presented in Section~\ref{Experimenting}.

We adhere to the following conventions and notation. 
Congruence images are sometimes indicated by overlining. 
Pre-images in $H =\langle g_1, \ldots , g_r\rangle\leq \SL(n,\Z)$ of 
elements of $\bar{H}$ written as words in the $\bar{g}_i$ are found 
by `lifting':
$\bar{g}_{k_1}^{m_1} \cdots \bar{g}_{k_s}^{m_s}$ has pre-image
$g_{k_1}^{m_1} \cdots g_{k_s}^{m_s}$.
The set of prime divisors of $a \in \mathbb{N}$ is denoted 
$\pi(a)$.  Throughout, $\F$ is a  field.

\section{Strong approximation and recognition of congruence images}
\label{SAandRecognitionOfImages}

The core idea of our approach to computing $\Pi(H)$ is to find 
all primes $p$ such that 
$\varphi_p(H)$ lies in a maximal subgroup of $\SL(n, p)$. 
Here we provide general methods for this purpose. 

\subsection{Large congruence images}

Let $H$ be infinite. Given a positive integer $k$, 
we find all primes $p$ such that $\varphi_p(H)$ has elements of order
greater than $k$
(cf.~\cite[Chapter 4]{Wehrfritz} and \cite[Section 3.5]{Recog}). 

Since a periodic linear group is locally finite, the finitely generated 
group $H$ has an element $h$ of infinite order. We can find $h$ quickly 
by random selection (see \cite[Section~4.2, p.~107]{Recog}, and 
the discussion in Subsection~\ref{randominfo} on
randomly selecting elements with specified properties).  
For $1\leq i\leq k$, let $m_i$ be the greatest common divisor of 
the non-zero entries of $h^i-1_n$, and let $l= \mathrm{lcm}(m_1,\ldots, m_k)$.
If $p \notin \pi(l)$ then $|\varphi_p(H)| > k$. 
For each $p \in \pi(l)$ we check whether $|\varphi_p(H)| < k$.
The preceding steps define a procedure ${\tt PrimesForOrder}$ that 
accepts $k$ and infinite $H \leq \SL(n, \Z)$,
and returns the (finite) set of primes $p$ such that $|\varphi_p(H)| < k$. 

We will also need the following.
\begin{lemma}\label{HInfinite}
Suppose that $\varphi_p(H) = \SL(n,p)$ for some prime $p$.
\begin{itemize}
\item[{\rm (i)}] If $n\geq 3$ then $H$ is infinite.
\item[{\rm (ii)}] If $n=2$ and $p\geq 3$ then $H$ is infinite.
\end{itemize} 
\end{lemma}
\begin{proof}
Theorem~A of \cite{Feitpreprint} states the largest order of a 
finite subgroup of $\GL(n,\Z)$.
In both cases (i) and (ii), this maximal order is less than 
$|\SL(n,p)|$. 
\end{proof}

\subsection{Irreducibility}
\label{irredsec}

This subsection recaps an argument from \cite[Section~3.2]{Density}.

We test whether $H\leq \SL(n, \Z)$ is absolutely irreducible
by computing a $\Q$-basis
$\mathcal{A} = \{A_1, \ldots , A_{m} \}$ of the enveloping 
algebra $\langle H \rangle_{\Q}$, where the $A_i$ are words over 
a generating set of $H$.  
If $m=n^2$ then $H$ is absolutely irreducible, and
$\varphi_p(H)$ is absolutely irreducible for any prime $p$ not 
dividing $\Delta := \mathrm{det}[\mathrm{tr}(A_iA_j)]$; 
here $\mathrm{tr}(x)$ is the trace of a matrix $x$. 
Hence we have the following. 
\begin{lemma}\label{LemmaAbsIrred}
If $H$ is absolutely irreducible then $\varphi_p(H)$ 
is absolutely irreducible for almost all primes $p$. 
\end{lemma}
If $p \, | \, \Delta$ then $\varphi_p(H)$ 
might be absolutely irreducible. Testing for this is 
the last step in ${\tt PrimesForAbsIrreducible}(H)$, 
which returns the set of all primes $p$ such that $\varphi_p(H)$ is 
not absolutely irreducible. 
Note that if $\bar{H}$ is absolutely irreducible
(e.g., $\bar{H} = \SL(n, p)$), and $\bar{\mathcal{A}}$
is a basis of $\langle \bar{H} \rangle_{\Z_p}$, then  
$\mathcal{A}$ is a basis of $\langle H \rangle_{\Q}$.

\subsection{Primitivity}

Next, we give conditions for the congruence image of an (irreducible) 
primitive subgroup of $\SL(n,\Z)$ to be imprimitive.
The main concern is prime $n$; in such degrees 
an irreducible linear group is either primitive
or monomial. 
\begin{lemma}\label{Lemma21}
If $H\leq \SL(n,\Z)$ is not solvable-by-finite
then $\varphi_p(H)$
is not monomial for almost all primes $p$. 
\end{lemma}
\begin{proof}
Since $H$ has a free non-abelian subgroup by the Tits alternative, 
given $k\geq 1$ there 
exist $g$, $h\in H$ such that 
$c:=[g^k, h^k] \neq 1_n$. 
Then $[\bar{g}^k, \bar{h}^k] = \bar{c} \neq 1_n$ for almost all primes
$p$. 
The lemma follows by taking $k$ to be the exponent of $\mathrm{Sym}(n)$.
\end{proof}
\begin{lemma}\label{Lemma21n}
For prime $n$, an infinite solvable-by-finite 
primitive (irreducible)
subgroup $H$ of $\, \SL(n, \Z)$ is solvable.
\end{lemma}
\begin{proof}
Let $K\unlhd H$ be solvable of finite index. 
Since $n$ is prime,  $K$ is scalar or irreducible. 
If $K$ were scalar then $H$ would be finite. 
Thus $K$ is irreducible. If $K$ were monomial
over $\Q$ then it would be finite once more; so $K$ is primitive.

Let $A$ be a maximal abelian normal subgroup of $K$. 
Then $A$ is irreducible, and $K = \langle A, g \rangle$ for some
$g$ because the field $\langle A \rangle_{\Q}$ is a cyclic 
extension of $\Q 1_n$ of degree $n$.
If $H$ normalizes $\langle A\rangle_\Q$ then $H$ is solvable. 
Suppose on the contrary that $hah^{-1} = bg^k$ for some $h\in H$, 
$a, b \in A$, and $k$ coprime to $n$.
Conjugation by $bg^k$ induces a $\Q$-automorphism of 
$\langle A \rangle_{\Q}$ that fixes $(bg^k)^n$. 
Hence $(bg^k)^n$ is scalar, implying that $a$ has finite order. 
But there is an infinite order element in $A$.
We conclude that $H$ must normalize $\langle A \rangle_{\Q}$.
\end{proof}

\begin{corollary}\label{Corollary21}
Let $n$ be prime. If $H \leq \SL(n,\Z)$ is infinite, non-solvable, 
and primitive, then $\varphi_p(H)$ is primitive for almost all 
primes $p$.
\end{corollary}

Given an input group $H$ that is 
not solvable-by-finite, ${\tt PrimesForMonomial}$ returns 
the set of primes $p$ such that $\varphi_p(H)$ is monomial. 
The proof of Lemma~\ref{Lemma21} furnishes a method to 
compute this finite set.
First we find $g$, $h\in H$ such that $[g^k, h^k] \neq 1_n$, 
where $k$ is the exponent of $\mathrm{Sym}(n)$.
(In our experiments $g$, $h$ are found by random selection; 
cf.~\cite{Aoun}, and see Subsection~\ref{randominfo}.) 
Let $d$ be the gcd of the non-zero entries of
$[g^k,h^k]-1_n$.  Then $\varphi_p(H)$ is non-monomial if $p\not \in \pi(d)$. 
Finally, we test whether $\varphi_p(H)$ is monomial for each 
$p\in \pi(d)$, using, e.g., \cite{gaprecog}. 

Although we can detect whether $H$ has a free non-abelian 
subgroup~\cite{Tits}, we do not have an algorithm to locate one. 
Indeed, as far as we know, the problem 
of deciding freeness of a finitely generated linear group is not known 
to be decidable.

\subsection{Solvability}

Zassenhaus's theorem~\cite[p.~136]{Supr} implies existence of a bound 
$\delta = \delta(n)$ on the derived length of solvable subgroups of
$\SL(n, \F)$ that depends only on $n$, not on $\F$. 
See, e.g., 
\cite[p.~136]{Supr} 
for an estimate of $\delta$ due to Dixon. 

Let $H\leq \SL(n,\Z)$ be non-solvable.
We sketch a procedure ${\tt PrimesForSolvable}(H,\delta)$ 
that returns the set of primes $p$ such that 
$\varphi_p(H)$ is solvable and $\varphi_p(H) \neq \SL(n, p)$. 
Take a non-trivial iterated commutator in $H$. As usual, we do this by 
random selection in $H$, or by lifting to $H$ from a (non-solvable) 
congruence image:
pick $[\bar{h}_1, \ldots, \bar{h}_{\delta + 1}] \neq 1_n$ in $\bar{H}$;
then $g = [h_1, \ldots, h_{\delta + 1}] \in H$ is as required. 
Let $d$ be the gcd of the non-zero entries of $g-1_n$. 
Then $\varphi_p(H)$ is non-solvable if $p\not \in \pi(d)$. 
Solvability of $\varphi_p(H)$ for $p\in \pi(d)$
can be tested using \cite{gaprecog}. 
We have proved the following.
\begin{lemma}
\label{LemmaBealsTits}
If $H$ is non-solvable then $\varphi_p(H)$ is non-solvable
for almost all primes $p$.
\end{lemma}

We get better bounds on derived length for irreducible groups in 
prime degree. 
\begin{lemma}
Let $n$ be prime. An irreducible solvable subgroup $G$ of 
$\GL(n,\F)$ has derived length $d\leq 6$. 
\end{lemma}
\begin{proof}
A monomial group $G$ is an extension of its subgroup 
of diagonal matrices by a solvable transitive permutation 
group of prime degree. Such permutation groups are metacyclic, 
so $d\leq 3$. Suppose that $G$ is primitive.By
\cite[Theorem~3.3, p.~42]{Wehrfritz}, there 
exists $E\unlhd G$ of derived length at most $2$, such that
$G/E$ is isomorphic to a subgroup of $\SL(2, n)$. 
Since $\delta(\SL(2, n)) \leq 4$ (see, e.g., \cite[\S 21.3]{Supr}), 
we get $d \leq 6$ as required.
\end{proof}
\begin{remark}
If $n = 2$, $3$ and $G\leq \SL(n,\F)$ then 
$d\leq 4$, $d\leq 5$, respectively. 
\end{remark}

\subsection{Isometry}

We say that $G\leq \GL(n,\F)$ is an \emph{isometry group} if
it preserves a non-degenerate bilinear 
(symmetric or alternating) form. On the other hand, 
since $\SL(2,\F) = \Sp(2,\F)$,
we say that $G$ is not an isometry group if $G$ does not preserve 
a non-degenerate bilinear form for $n>2$.
\begin{lemma}\label{Lemma24}
Let $G\leq \GL(n, \F)$ be absolutely irreducible. Then $G$ is an
isometry group if and only if 
$\mathrm{tr}(g) = \mathrm{tr}(g^{-1})$ for all $g \in G$.
\end{lemma}
\begin{proof}
Suppose that $\mathrm{tr}(g) = \mathrm{tr}(g^{-1})$ for all $g \in G$. 
As their characters are equal, the identity and contragredient 
representations of $G$ are therefore equivalent;
i.e., $g = \Phi (g^\top)^{-1} \Phi^{-1}$ for some 
$\Phi\in \allowbreak \GL(n,\F)$.
Rearranging this equality, we see that $G$ preserves the form
with matrix $\Phi$.
\end{proof}

The procedure ${\tt PrimesForIsometry}$ accepts an 
absolutely irreducible subgroup $H$ of $\SL(n,\Z)$ that is not an isometry group.
It selects $h\in H$ such that $a:=\mathrm{tr}(h) - \mathrm{tr}(h^{-1}) \neq 0$,  
and (using \cite{gaprecog}) 
returns those $p \in \pi(a)$ such that $\varphi_p(H)$ is an isometry group.

We will need to check not only whether
a congruence image of $H$ preserves a form, but whether it 
lies in the similarity group generated by 
a full isometry group and all scalars.
This is achieved with ${\tt PrimesForSimilarity}(H)$, which
selects $h=[h_1,h_2]\in H$ such that 
$a:=\mathrm{tr}(h) - \mathrm{tr}(h^{-1}) \neq 0$. 
Clearly $\varphi_p(H)$ is in a similarity group only if $p \in \pi(a)$.
\begin{lemma}
\label{BilinearReduction} 
Suppose that $H\leq \SL(n,\Z)$ is absolutely irreducible and 
not an isometry group. Then for almost all primes $p$, 
$\varphi_p(H)$ does not lie
in a similarity group over $\Z_p$.
\end{lemma}

\section{Algorithms for strong approximation}\label{AlgorithmsSA}

We proceed to formulate an algorithm that realizes strong approximation
in prime degree $n$. That is, the algorithm computes $\Pi(H)$ for any
dense input $H\leq \SL(n, \Z)$. We also compute $\Pi$ 
for dense subgroups of $\SL(2n, \Z)$  containing a transvection. 

\subsection{Density in prime degree}\label{DenseInPrimeDegree}

For the entirety of this subsection, $n$ is prime. 

By \cite{Aschbacher84} (cf \cite[p.~397]{LubotzkySegal}), 
the set $\mathcal{C}$ of maximal subgroups of $\SL(n, p)$ is a union 
of certain subsets 
$\mathcal{C}_1, \ldots , \mathcal{C}_9$.
For each $i$, we determine all primes 
$p$ such that $\varphi_p(H)$ could be in a group in $\mathcal{C}_i$. 
Hence, we provide criteria for $H$
to surject onto $\SL(n, p)$ for almost all primes $p$. 
These conditions turn out to be equivalent to 
density. They constitute the background of our main algorithm 
and obviate any need to test density of the input 
(as in, say, \cite[Section~5]{Density}). 

We start with an auxiliary statement for $\mathcal{C}_9$
(called class ${\mathcal S}$ in~\cite[Chapter 8]{Bray}).
\begin{lemma}
\label{primdegrep}
There is a bound in terms of $n$ on the order of subgroups 
of $\SL(n,p)$ that are contained solely in groups in 
${\mathcal C}_9$.
\end{lemma}
\begin{proof}
Suppose that $U\le\SL(n,p)$ lies only in ${\mathcal C}_9$ and not in
${\mathcal C}_i$ for $i\neq 9$. 
The perfect residuum $U^\infty$ (i.e., the 
last term of the derived series of $U$)
is therefore a simple absolutely irreducible subgroup of $\SL(n, p)$. 
If we show that the order of $U^{\infty}$ is bounded, then 
$U\le\mathrm{Aut}(U^{\infty})$ also has bounded order. Thus, without 
loss of generality,  $U=U^{\infty}$ from now on.

Prime degree faithful representations of quasisimple 
groups are classified in \cite[Theorem~1.1]{MalleZaleskii}. 
The orders of the groups in classes (10)--(27) of this classification
are bounded absolutely (i.e., by a bound not depending on $n$ or $p$). 
The orders of groups in classes (2)--(9) are bounded by a function 
of $n$.

Class (1) groups are of Lie type in characteristic $p$ 
in the Steinberg representation~\cite{SteinbergRep}, whose degree $n$ 
is the $p$-part of the group order. For each class of groups of 
Lie type ${\mathcal G}_m(p)$, this $p$-part is $p^a$ with 
$a \leq 1$ for only finitely many values of $m$. So  class (1) is 
finite for prime $n$.

Finally we come to the case excluded by \cite[Theorem 1.1]{MalleZaleskii},
namely $U/Z(U)\cong \mathrm{Alt}(m)$ for $m>18$. As a consequence
of~\cite{James83,KleshchevTiep12}, there are only finitely many 
degrees $l$ such that $\mathrm{Sym}(l)$ and thus 
$\mathrm{Alt}(l)$ has a faithful (projective) 
representation of degree $m$. 
\end{proof}

The main procedure, ${\tt PrimesForDense}(H)$,  combines
the subsidiary procedures of Section~\ref{SAandRecognitionOfImages}.
Its output is the union of
\begin{itemize}
\item[] \vspace{-12pt}
\item ${\tt PrimesForAbsIrreducible}(H)$ 
\item[] \vspace{-12pt}
\item ${\tt PrimesForMonomial}(H)$
\item[] \vspace{-12pt}
\item ${\tt PrimesForSolvable}(H, \delta)$, where $\delta$ is a bound on the
derived length of a solvable linear group of degree $n$
\item[] \vspace{-12pt}
\item ${\tt PrimesForSimilarity}(H)$
\item[] \vspace{-12pt}
\item ${\tt PrimesForOrder}(H, k)$ where $k$ is a bound on element orders
 for groups of degree $n$ in ${\mathcal C}_6 \cup {\mathcal C}_9$.
\item[] \vspace{-12pt}
\end{itemize}
\begin{theorem}
Assuming termination for input $H$, ${\tt PrimesForDense}(H)$ 
returns $\Pi(H)$.
\end{theorem}
\begin{proof}
Each prime returned must lie in $\Pi(H)$.
Conversely, let $p$ be a prime such that $\varphi_p(H)\neq \SL(n,p)$.
Then $\varphi_p(H)$ is in a group in some ${\mathcal C}_i$, 
$1\leq i\leq 9$. For each $i$, we show that 
(at least) one of the subsidiary procedures returns $p$.

 ${\mathcal C}_1$:
here $\varphi_p(H)$ is reducible, so $p$ is 
returned by 
${\tt PrimesForAbsIrreducible}(H)$.

${\mathcal C}_2$:
$p$ is returned by ${\tt PrimesForMonomial}(H)$.

 ${\mathcal C}_3$:
for prime $n$, the stabilizers of extension fields
are solvable, so $p$ is
returned by ${\tt PrimesForSolvable}(H, \delta)$.

 ${\mathcal C}_4$,
${\mathcal C}_7$: since
the degree of a tensor product is the product of the factor degrees,
and $n$ is prime, these classes are empty. 

 ${\mathcal C}_5$
is empty over fields of prime size.

 ${\mathcal C}_6$
consists of groups whose structure
depends on $n$ but not on $p$~\cite[Section~2.2.6]{Bray}. The number of such
groups (and thus the largest order of an element in any one of them) is 
bounded, and so ${\tt PrimesForOrder}(H, k)$ returns $p$.

 ${\mathcal C}_8$:
the groups in this class preserve a form modulo $Z(\SL(n,p))$. Hence
the derived group of $\varphi_p(H)$ preserves a form and $p$ is returned by
${\tt PrimesForSimilarity}(H)$.

 ${\mathcal C}_9$:
by Proposition~\ref{primdegrep}, the number of groups in this 
class is finite.
Thus (as with ${\mathcal C}_6$)
${\tt PrimesForOrder}(H, k)$ returns $p$.
\end{proof}
\begin{remark}
Using {\sf GAP} and tables in \cite[Chapter~8]{Bray}, 
we can calculate bounds on the order of groups 
in ${\mathcal C}_6 \cup {\mathcal C}_9$
(and hence bounds on their element orders) for small $n$. 
For $n=2$, $3$, $5$, $7$, $11$,
these bounds are $10$, $21$, $60$, $84$, $253$, 
respectively.
\end{remark}
\begin{remark}
${\tt PrimesForDense}$ simplifies in small degrees. 
If $n \leq 3$ then
the groups in ${\mathcal C}_2$ are solvable, so 
${\tt PrimesForSolvable}$ overrides ${\tt PrimesForMonomial}$. 
In degree $2$, ${\tt PrimesForSimilarity}$ is also redundant.
\end{remark}

If ${\tt PrimesForDense}(H)$ terminates then
$\Pi(H)$ is finite, i.e., $H$ is dense~\cite[p.~3650]{Rivin}. 
Next we prove the converse.
This leads to a characterization of density in $\SL(n,\Z)$.
\begin{lemma}\label{LemmaA}
If $H$ is irreducible, not solvable-by-finite, and 
not an isometry group, then $\Pi(H)$ is finite.
\end{lemma}
\begin{proof}
Each constituent output set is finite by 
Lemmas~\ref{LemmaAbsIrred}, \ref{Lemma21}, \ref{LemmaBealsTits}, 
\ref{BilinearReduction}, and \ref{primdegrep}.
\end{proof}

\begin{lemma}\label{LemmaB}
If $H$ is
infinite, non-solvable, primitive, and not an isometry group,
then  $\Pi(H)$ is finite. 
\end{lemma}
\begin{proof}
As the previous proof, but relying on Corollary~\ref{Corollary21}
instead of Lemma~\ref{Lemma21}.
\end{proof} 
\begin{lemma}\label{Lemma31nn}
Suppose that $\varphi_p(H) = \SL(n, p)$ for some prime $p$, where 
$p > 3$ if $n=2$. 
Then $H$ is infinite, non-solvable, and 
primitive. Furthermore, $H$ is not an isometry group.
\end{lemma}
\begin{proof}
Since $\SL(n, p)$ is absolutely irreducible and 
non-solvable, the same is true of $H$. A monomial subgroup of $\SL(n,\Z)$ 
cannot 
surject onto $\SL(n,p)$ because it has an abelian normal subgroup
whose index is too small. The remaining assertion 
follows from Lemmas~\ref{HInfinite} and \ref{Lemma24}.
\end{proof}

Lemmas~\ref{LemmaB} and \ref{Lemma31nn} yield
\begin{corollary}
\label{Corollary32}
If $\varphi_q(H) = \SL(n, q)$ for one prime $q > 3$, 
then $\varphi_p(H) = \SL(n, p)$ for almost all primes $p$.
\end{corollary}
\begin{remark}
Corollary~\ref{Corollary32} should be compared with 
\cite[~p.~396]{LubotzkySegal}, \cite[Proposition~1]{Lubotzky97}, 
and \cite{Weigel}.
\end{remark}

\begin{corollary}
The following are equivalent.
\begin{itemize}
\item[{\rm (i)}] $H$ is dense.
\item[{\rm (ii)}] $H$ surjects onto $\SL(n,p)$ 
modulo some prime $p$, where $p > 3$ if $n=2$.
 \item[{\rm (iii)}] $H$ is infinite, 
non-solvable, primitive,  and not an isometry group. 
\item[{\rm (iv)}] $H$ is irreducible, not solvable-by-finite,
and not an isometry group.
\end{itemize}
\end{corollary}

\begin{remark}
Let $n = 2$. Then $H$ is dense if and only if $H$ is not
solvable-by-finite; which is equivalent to $H$ being infinite and 
non-solvable.
\end{remark}

To round out the subsection, we give one more set of criteria for
density in odd prime degree.
\begin{lemma}\label{Lemma32n} 
Let $n > 2$. If $H$ contains an irreducible element and is not 
solvable-by-finite then $H$ is dense.
\end{lemma}
\begin{proof}
We appeal to Lemma~\ref{LemmaA}. Let 
$h\in H$ be irreducible. 
Suppose that $H$ preserves a 
form with (symmetric or skew-symmetric) matrix $\Phi$.
Then $x \mapsto \Phi x^\top \Phi^{-1}$
defines a $\Q$-automorphism of $\langle h \rangle_{\Q}$ 
of order $2$. But 
$\langle h \rangle_{\Q}$ is a field extension 
of odd degree $n$. Hence $H$ is not an isometry group. 
\end{proof}

\begin{corollary}\label{Corollary32n}
For $n > 2$, a finitely generated subgroup of $\SL(n,\Z)$ is dense 
if and only if it contains an irreducible element and is not 
solvable-by-finite. 
\end{corollary}

\begin{remark}\label{remark31}
Lemma~\ref{LemmaB} allows us to replace `not solvable-by-finite'
in Lemma~\ref{Lemma32n} and Corollary~\ref{Corollary32n}  
by `infinite  non-solvable primitive', or by  `infinite non-solvable' 
if $n=3$
(cf.~\cite[p.~415]{LongReidI}, \cite[Theorem 2.2]{LongReidII}).
\end{remark}

\subsection{Algorithms for groups with a transvection} 
In \cite[Section 3.2]{Density} we gave a straightforward 
procedure $\tt PrimesForDense$
to compute $\Pi(H)$ if $H$ is dense in $\SL(2n+1, \Z)$
or $\Sp(2n, \Z)$ and contains a known transvection.  
The case $H\leq \SL(2n, \Z)$ was left open. Now we close that gap.
\begin{lemma}\label{LemmaTrans}
Suppose that $H \leq \SL(2n, \Z)$ contains a transvection $t$. 
Then $H$ is dense
if and only if $N := \langle t \rangle^{H}$ is absolutely irreducible 
and $\mathrm{tr}(h) \neq \mathrm{tr}(h^{-1})$ for some $h$ in $N$.
\end{lemma}
\begin{proof}
Suppose that $H$ is dense. Then $N$ is absolutely irreducible 
by \cite[Corollary 3.5]{Density}. If $\mathrm{tr}(h) = \mathrm{tr}(h^{-1})$ 
for all $h \in N$,  then by Lemma~\ref{Lemma24} there is a 
form with matrix $\Phi$ such that $h \Phi h^\top = \Phi$.
Since $N \unlhd H$ and $N$ is
absolutely irreducible, $h \Phi h^\top = \alpha \Phi$ 
for all $h \in H$ and some $\alpha \in \Q$ (see, e.g., 
\cite[Lemma~1.8.9, p.~41]{Bray}).
This contradicts density of $H$.
 
Now suppose that $N$ is absolutely irreducible and 
$\mathrm{tr}(h) \neq \mathrm{tr}(h^{-1})$ for some $h\in N$. 
Then $\varphi_p(N)$ is absolutely irreducible and
$\varphi_p(\mathrm{tr}(h)) \neq \varphi_p(\mathrm{tr}(h^{-1}))$
for almost all primes $p$. 
So there are $p > 3$ and $g\in \varphi_p(N)$ 
such that $\varphi_p(N)$ is absolutely irreducible
and $\mathrm{tr}(g) \neq \mathrm{tr}(g^{-1})$.
Since $\varphi_p(N)$ is generated by transvections, the theorem of
\cite[p.~1]{ZalesSeregkin} implies that $\varphi_p(N)=\SL(2n, p)$
or $\Sp(2n, p)$.
Since the latter possibility is ruled out by Lemma~\ref{Lemma24}, we
must have $\varphi_p(H) = 
\SL(2n, p)$ and so $H$ is dense
(see \cite[Proposition~1]{Lubotzky97}).
\end{proof}

The procedure ${\tt PrimesForDense}(H, t)$, based on
Lemma~\ref{LemmaTrans}, accepts dense $H\leq \SL(2n, \Z)$ 
containing a transvection $t$, and returns $\Pi(H)$. It combines
${\tt PrimesForAbs}$-${\tt Irreducible}(N)$ and ${\tt PrimesForIsometry}(N)$, 
checking whether 
$\varphi_p(H) = \SL(2n, p)$ for each $p$ in the union of their outputs. 
See \cite[Section 3]{Density} for an
algorithm to compute a basis of 
$\langle N \rangle_{\Q}$ without computing (a full generating set of) 
the normal closure $N$. 
Similarly, the application of ${\tt PrimesForIsometry}$ does not
require computing $N$, and just randomly selects $h \in N$ 
such that $\mathrm{tr}(h) \neq \mathrm{tr}(h^{-1})$.

\subsection{General considerations}
\label{randominfo}

We comment further on the operation of our algorithms.

When selecting (pseudo-)random elements of $\SL(n, \Z)$ 
for some subprocedures, we seek just one element with a nominated property.
These will be plentiful in dense subgroups.
Hence we do not aim for any semblance of a uniform distribution 
(cf.~\cite[Section~5]{Rivin2}),
but randomly take words of length $5$ in the given generators. If these 
repeatedly fail to have the desired property then we gradually increment 
the word length. We do not have a theoretical bound on 
the runtime for this process; but in practice we observe that it is
very fast.

At the start of the calculation we also select (e.g., by computing the
orders, or invoking composition tree on images of $H$ modulo different 
primes~\cite{gaprecog}) 
a prime $p_0>3$ such that
$\varphi_{p_0}(H)=\SL(n,p_0)$. The properties of
elements that we are seeking may then be maintained modulo $p_0$. That
is, instead of searching in $H$, we search for an element
$\bar h$ in $\varphi_{p_0}(H)$ that has the desired properties (over $\Z_{p_0}$)
and lift to the pre-image $h\in H$.

Each of the subsidiary procedures for ${\tt PrimesForDense}(H)$
returns a positive integer $d$ divisible by every prime $p$ 
such that $\varphi_p(H)$ is in the relevant
class of maximal subgroups of $\SL(n,p)$. However, $d$ can  
have prime factors not in $\Pi(H)$.
Furthermore, these factors
might be so large as to make factorization of $d$ impractical, or 
make the test of the congruence image overly expensive.
Thus we do not factor $d$ fully,
but only attempt a cheap partial factorization (e.g., by trial
division and a Pollard-$\rho$ algorithm). If $d$ does not factorize, or has
large prime factors (magnitudes larger than the entries of
the input matrices), then we compute another positive integer $d'$ using the
same algorithm but with different choices of random elements, and replace 
$d$ by $\gcd(d,d')$. 

Our computational realization of strong approximation stands in
contrast to Breuillard's quantitative version~\cite[Theorem~2.3]{Breuillard}.
His bound on the primes that can appear in $\Pi(H)$ is not explicit. We
compute the entire set $\Pi(H)$; and can do so quickly, as shown in the next 
section.

\section{Experimenting with low-dimensional dense subgroups}
\label{Experimenting}

In this section we present experimental results obtained from
our {\sf GAP} implementation of the algorithms.
We demonstrate the practicality of our software and how it 
can be used to obtain important 
information about groups.
In particular we describe all congruence images, as explained
in the next subsection. 

\subsection{Computing all congruence images}
\label{ComputingAllCongruenceImages}
Let $H\leq \SL(n,\Z)$ be dense. 
As in \cite[Section~2.4.1]{Density}, let $\tilde{\Pi}(H) = \Pi(H)\cup \{ 2\}$
if $\varphi_2(H) = \SL(n, 2)$  and $\varphi_4(H) \neq 
\SL(n, 4)$; let $\tilde{\Pi}(H) = \Pi(H)$ otherwise. The disparity 
between $\tilde{\Pi}(H)$ and $\Pi(H)$ can arise only when $n\leq 4$, and
$M(H)$ is even but $2\not \in \Pi(H)$.
By \cite[Theorem 2.18]{Density}, $\tilde{\Pi}(H)= \pi(M(H))$.
If $n>2$ then $\varphi_k(H) = \varphi_k(\mathrm{cl}(H))$ for all $k$; 
so $\tilde{\Pi}(H) = \tilde{\Pi}(\mathrm{cl}(H))$. We 
may therefore assume that $H$ is arithmetic, of 
level $M$.
Let $a =\gcd(k,M)$, so $k=abc$, $\pi(b)\subseteq \pi(a)$, and 
$\gcd(c,a) = 1$. Then
$\varphi_k(H) \cong H/(H\cap \Gamma_k)\cong H\Gamma_k/\Gamma_k$ is 
a subgroup of $\Gamma_{ab}/\Gamma_k \times \Gamma_c/\Gamma_k$.
It is not difficult to show that 
$\varphi_k(H)$  splits as a direct product 
of $\Gamma_{ab}/\Gamma_k$ with 
$Q:=((H\Gamma_k)\cap \Gamma_c)/\Gamma_k$.
Since $\Gamma_{ab}/\Gamma_k\cong
\SL(n,\Z)/\Gamma_c$, this expresses
$\varphi_k(H)$ as a direct product of
$Q$ with a subgroup isomorphic to $\SL(n,\Z_c)$.
Hence the task in describing all congruence
images of $H$ boils down to computing with the quotient $Q$
of $\varphi_k(H)$ in $\SL(n,\Z_k)$; in effect, ranging over all 
divisors of $M$.
If $n=2$ then the congruence subgroup property does not hold, and 
we can only handle $k=p$ prime. 

In some of the examples below we describe
the congruence quotient modulo the level $M$, exhibiting which parts of its
structure arise for various prime powers. We give this as an ATLAS-style
composition structure \cite{Atlas} (separating composition factors by 
dots; cf.~\cite{Bray}), marked
up to show the prime powers for which each factor first arises.
We emphasize that these have been generated `semiautomatically' 
using {\em some} composition series that refines the congruence structure, not
necessarily the best possible series.
One example from Table~\ref{Table2} is a group of level $3^45{\cdot}19$
with quotient structure
\[
\ppqf{3^4}{3^4}.\ppqf{3^3}{3^3}.\ppqf{3^3}{3^2}.\ppqf{5^2.2.2.2}{5}.
\ppqf{3.3}{3,5}.\ppqf{L_2(19)}{19}
\]
In the standard notation $L_m(q) := \PSL(m,q)$,
this has congruence image $L_2(19)$ modulo $19$, which is a simple direct
factor not interacting with the other primes. The quotient modulo $3$ has
structure $3.3$ (and is almost certainly 
the group
$3^2$). The quotient modulo $5$ is $5^2.2.2.2.3.3$, forming a subdirect
product with the quotient of order $3$ in which the full factor $3.3$ is glued
together. Modulo $9$ the group possesses a factor $3^3$ (of the
possible $3^{3\cdot 3-1}=3^8$), modulo $27$ another factor $3^3$, and modulo
$81$ a factor $3^4$. (Since $3^4$ is the prime power dividing the level, the
quotient modulo $243$ would contain a full $3^8$.)  The structural
analysis in~\cite[Section 2]{Density} proves that the exponent 
for $p^{i+1}$ cannot be
smaller than the exponent for $p^i$.
The name indicates all proper prime powers dividing the level.
Thus `empty' factors $\ppqf{\ }{p^a}$ are possible if the group 
has no elements on that level.

\vspace{4pt}

Experimental results are displayed in Tables~\ref{Table2} and \ref{Table3} 
(writing $A_m$ for $\mathrm{Alt}(m)$ and $S_m$ for $\mathrm{Sym}(m)$). 
Our actual implementation computes $\tilde{\Pi} = \pi(M)$ rather than 
$\Pi(H)$. 
We do not state $\Pi(H)$; as noted above, this set almost always coincides
with $\pi(M)$.

Experiments were performed on
a 2013 MacPro with a 3.7 GHz Intel Xeon E5 utilizing up to 8GB of memory. 
The software can be accessed \href{http://www.math.colostate.edu/~hulpke/arithmetic.g}{here}.
Some documentation~\cite{GAPDoc} is also available.

\subsection{Low-dimensional dense subgroups}

Our examples in this subsection
come from a family of integral representations of 
finitely presented 
groups, as defined in \cite{LongReidI,LongReidII,LongNew}.
For each test group $H$ we compute $\Pi(H)$, incidentally justifying  
that $H$ is dense. Thereafter we compute  
$M(H)$, $|\SL(n, \Z) : H|$, 
and the congruence quotients of $H$. 
\subsubsection{The fundamental group of the figure-eight knot complement}
\label{figure8}
Adopting the notation of \cite[p.~414]{LongReidI}, let
\[
\Gamma := \langle x,y,z \mid zxz^{-1} = xy, zyz^{-1} = yxy \rangle;
\]

\noindent this is the fundamental group of the figure-eight knot complement. 
Put $F = \langle x, y \rangle$. 
In \cite{LongReidI}, two families of representations $\beta_T$, 
$\rho_k$ of $\Gamma$ in $\SL(3, \Z)$ were constructed.  
Section~4 of \cite{Density} reports on experiments with $\beta_T$ 
for a range of $T$ and $\rho_k$ for $k = 0, 2, 3, 4, 5$. The groups 
$\rho_k(\Gamma)$, $\rho_k(F)$ for $k \neq 0, 2, 3, 4, 5$ are 
of special interest (see \cite[Section~5]{LongReidI}). 
However, neither the methods of \cite{LongReidI} nor 
those of \cite{Density} facilitate proper study of $\rho_k$ 
for such $k$. 

We have
\[
\rho_k(x) =  
\left(\begin{array}{ccc} 1 & -2 & 3
\\
0 & k & -1-2k
\\
0 & 1 & -2 \end{array} \right), \quad \rho_k(y)  
= 
\left(\begin{array}{ccc} -2-k & -1 & 1
\\
-2-k & -2 & 3
\\
-1 & -1 & 2 \end{array} \right),
\]
\[
\rho_k(z) 
= 
\left(\begin{array}{ccc} 0 & 0 & 1
\\
1 & 0 & -k
\\
0 & 1 & -1-k \end{array} \right).
\]


\noindent 
The results of experiments with 
$\rho_k(\Gamma)$ and $\rho_k(F)$ 
are collected in Table~\ref{Table2}; here $M$ is
the level (which turns out to be the same for both $\Gamma$ and $F$),
$\mbox{Index}_\Gamma$ is $|\SL(3, \Z) : \mathrm{cl}(\rho_k(\Gamma))|$, and 
$\mbox{Index}_{\Gamma,F}$ is $|\mathrm{cl}(\rho_k(\Gamma)): \mathrm{cl}(\rho_k(F))|$.
The last column is the congruence image of $\rho_k(F)$ modulo $M$.
For $k=1$, $6$, $10$, the groups surject modulo $2$ but not 
modulo $4$. 

Determination of the exceptional primes was instantaneous.
The time to calculate level and index increased roughly with level,
from a few seconds for $k=1$ to about $15$ minutes for $k=20$.

\medskip

\begin{table}[htb]
\begin{tabular}{r|r|r|r|r}
$k$
&$M$ \phantom{xx}
&$\mbox{Index}_\Gamma$\phantom{xxxxxx}
&Index$_{\Gamma,F}$
&StructureF \phantom{XXXXX}
\\
\hline
$1$
&$2^{2}3^{4}$
&$2^{10}3^{15}13$
&$2^{2}$
&$\ppqf{3^4}{3^4}.\ppqf{3^3}{3^3}.\ppqf{\ }{2^2}.\ppqf{3^3}{3^2}.\ppqf{3.3}{3}.%
\ppqf{L_3(2)}{2}$
\\
$6$
&$2^{2}31
{\cdot}43$
&$2^{10}3^{3}7
{\cdot}43^{2}331
{\cdot}631$
&$2
{\cdot}3
{\cdot}5$
&$\ppqf{\ }{2^2}.\ppqf{31.31.2}{31}.\ppqf{L_2(43)}{43}.\ppqf{L_2(31)}{31}.%
\ppqf{L_3(2)}{2}$
\\
$7$
&$3^{4}5
{\cdot}19$
&$2^{6}3^{17}5
{\cdot}13
{\cdot}19^{2}31
{\cdot}127$
&$2^{2}3^{2}$
&$\ppqf{3^4}{3^4}.\ppqf{3^3}{3^3}.\ppqf{3^3}{3^2}.\ppqf{5^2.2.2.2}{5}.\ppqf{3.3\
}{3,5}.\ppqf{L_2(19)}{19}$
\\
$10$
&$2^{2}3^{4}11
{\cdot}37$
&$2^{14}3^{16}7^{2}13
{\cdot}19
{\cdot}37^{2}67$
&$2^{2}3^{2}5$
&\begin{minipage}[t]{4.5cm}
$\ppqf{3^4}{3^4}.\ppqf{3^3}{3^3}.\ppqf{\ }{2^2}.\ppqf{3^3}{3^2}.
\ppqf{11.11.2}{11}.\ppqf{3}{3}.\ppqf{3}{3,11}$
$\hphantom{XX}.\ppqf{L_2(37)}{37}.\ppqf{L_2(11)}{11}.\ppqf{L_3(2)}{2}$
\end{minipage}
\\
$15$
&$229
{\cdot}241$
&$2^{6}3^{3}5
{\cdot}97
{\cdot}181
{\cdot}241^{2}19441$
&$2
{\cdot}3
{\cdot}19$
&$\ppqf{229.229.2}{229}.\ppqf{L_2(241)}{241}.\ppqf{L_2(229)}{229}$
\\
$20$
&$409
{\cdot}421$
&$2^{4}3^{3}5
{\cdot}7
{\cdot}421^{2}55897
{\cdot}59221$
&$2^{2}3
{\cdot}17$
&$\ppqf{409.409.2}{409}.\ppqf{L_2(421)}{421}.\ppqf{L_2(409)}{409}$
\\
\end{tabular}

\medskip

\caption{}
\label{Table2}
\end{table}

\subsubsection{Triangle groups} 
\label{Delta}
Next we look at triangle groups 
$\Delta(p, q, r) = \langle a, b \mid a^p = b^q = (ab)^r = 1 \rangle$.

In \cite{LongReidII}, representations of 
$\Delta(3, 3, 4)$ in $\SL(3, \Z)$ are defined  by
\[
a \mapsto a_1 = {\footnotesize \left( \begin{array}{ccc}
0 & 0 & 1\\
1 & 0 & 0 \\
0 & 1 & 0\end{array}\right)}, \quad 
b \mapsto b_1(t) = {\footnotesize \left( \begin{array}{ccc}
1 & 2-t+t^2 & 3+t^2\\
0 & -2+2t-t^2 & -1+t-t^2\\
0 & 3 - 3t + t^2 &(-1 + t)^2
\end{array}\right)}.
\] 
These representations are faithful for all $t \in \mathbb{R}$, 
and if $t \in \Z$ then the images are dense 
and non-conjugate for different $t$~\cite[Theorem 1.1]{LongReidII}. 
If $t = 1$ then the group is conjugate to the one constructed by 
Kac and Vinberg~\cite[p.~422]{LongReidI}. 
Put $H_1(t) = \langle a_1, b_1(t) \rangle$. 

In \cite[p.~8]{LongReidII}, the following 
faithful 
dense representations $H_2(t) = \langle a_2(t), b_2 \rangle$ of 
$\Delta(3,4,4)$ were constructed:
\[
a \mapsto a_2(t) = 
{\footnotesize \left( \begin{array}{ccc} 
1 & 4+3t^2/4 & 3(6-t+t^2)/2\\
0 & -(4+t+t^2)/2 & -3-t^2\\
0 & (4+2t+t^2)/4 & (2+t+t^2)/2
\end{array}
\right)}, 
\quad
b \mapsto b_2 = {\footnotesize \left( \begin{array}{ccr}
 0 & 0 & 1 \\ 
1 & 0 & -1\\
0 & 1 & 1\end{array}
\right)}.
\]

In \cite[p.~13]{LongNew}, faithful representations 
of $\Delta(3,3,4)$ in $\SL(5,\Z)$ are defined by
\[
a \mapsto a_3(k) = 
{\footnotesize \left( \begin{array}{ccccc} 
1 & 0 & -3-2k-8k^2 & -1+10k+32k^3 & -5-16k^2\\
0 & 4(-1+k) & -13 - 4k & 3+16(1+k)^2 & -4+16k \\
0 & 1-k+4k^2 & 3-2k+8k^2 & -2(1+3k+16k^3) & 3+16k^2\\
0 &  k & 2k & 1-2k-8k^2 & 1+4k\\ 
0 & 0 &  3k & 3(-1+k-4k^2) & -2
\end{array}
\right)}, 
\]
\[
b \mapsto b_3(k) = 
{\footnotesize \left( \begin{array}{ccccc} 
0 & 0 & -3-2k-8k^2 & -1+10k+32k^3 & -5-16k^2\\
0 &1 & 3+4k & -13-8k-16k^2 & 4-16k\\
0 & 0 & -2(1+k+4k^2) & 6k+32k^3 & -3-16k^2\\
1 & 0 & -2(1+k) & -1+2k+8k^2 & -1-4k\\
2k & 0 & 1-2k &-4k & 1\end{array} \right)}.
\]
As $k$ ranges over $\Z$, the 
 $H_3(k) = \langle a_3(k), b_3(k) \rangle$ are dense and
pairwise non-conjugate.

It is known that $H_1(t)$, $H_2(t)$, $H_3(k)$ are 
thin~\cite{LongReidII,LongNew}. 
For each of these groups we computed its level $M$ and the index 
of its arithmetic closure in $\SL(n, \Z)$ for several
values of the parameters. See Table~\ref{Table3}. 

For $t\equiv1\pmod{4}$, the $H_1(t)$ as far as we tested
surject onto $\SL(3,2)$ but not onto $\SL(3,\Z_4)$.

Runtimes for degree $3$ groups 
were consistent with 
the previous example. In degree $5$,
identification of primes was again instantaneous, while the 
calculation of level and index took
about $6$ minutes for $H_3(0)$ and $20$ minutes for $H_3(3)$. So
we did not try larger $k$.

\subsubsection{Random generators}
We constructed subgroups of $\SL(n,\Z)$ for $n=3$, $5$ generated by a
pair of pseudo-random matrices (via the {\sf GAP} command ${\tt RandomUnimodularMat}$).
More than half of the groups so generated
surject onto $\SL(n,p)$ modulo all primes $p$ (and modulo $4$). 
We attempted to verify whether each group is arithmetic by expressing 
its generators as words in standard generators of $\SL(n, \Z)$ and 
running a coset enumeration with the presentation from~\cite{Steinberg85}. 
As the enumeration never terminated, we suspect that these groups
are not arithmetic
(note that a random finitely generated subgroup of $\SL(n, \Z)$ is 
likely to be thin~\cite{FuchsRivin,Rivin}).

\subsubsection{Further experimentation}
Comparing congruence images with
finite quotients (obtained, e.g., by the low-index algorithm of
\cite[Section~5.4]{HoltHandbookCGT}) may help to decide whether 
a dense representation of a finitely presented group is faithful, or
justify that a group is thin. 
For example, low-index calculations with the finitely presented group
$\Gamma$ as in 
Subsection~\ref{figure8} expose quotients (such as $\mathrm{Sym}(23)$, 
$\mathrm{Sym}(29)$, $\mathrm{Alt}(11)\wr C_2$, to name just a
few) that cannot be congruence images of any $\rho_k(\Gamma)$, as they 
do not have representations of suitably small degree.
Thus $\rho_k$ cannot be faithful on $\Gamma$ if
$\rho_k(\Gamma)$ is arithmetic 
(cf.~\cite[Question~5.1]{LongReidI}).
This fact has a clear explanation: $F$ is free and normal in $\Gamma$;
hence a representation of $\Gamma$ in $\SL(3, \Z)$ is arithmetic 
precisely when its restriction on $F$ is 
arithmetic~\cite[p.~420]{LongReidI}; 
but any virtually free group cannot have a faithful arithmetic 
representation in $\SL(n, \Z)$ for $n >2$.

To illustrate another potential application of our algorithms, we
show that faithful dense representations 
of the triangle groups $\Delta(3,3,4)$, $\Delta(3,4,4)$ 
in $\SL(3,\Z)$ or $\SL(5, \Z)$ are not arithmetic; this includes 
$H_1(t)$, $H_2(t)$, $H_3(k)$ as in Subsection~\ref{Delta} 
(cf.~\cite{LongReidII,LongNew}).
Indeed, $\Delta(3,3,4)$ and $\Delta(3,4,4)$ each have a quotient 
isomorphic to $\mathrm{Alt}(20)$. This is not a congruence 
quotient of an arithmetic group in $\SL(3,\Z)$ or 
$\SL(5, \Z)$, because $\mathrm{Alt}(20)$ does not have a faithful 
representation in $\SL(3,p)$ or $\SL(5,p)$ for any $p$.

We also use this example to compare the capability of our algorithm 
with that of the low-index algorithm. 
Congruence quotients of $\rho_k(\Gamma)$ 
(modulo any integer $m>1$, including $m$
not dividing the level) produced by our algorithms expose quotients of 
$\Gamma$ (such as $\SL(n,p)$ for large $p$)  that are infeasible to find 
through a low-index computation, because these groups do not have a faithful 
permutation representation of sufficiently small degree. 
Using a homomorphism search~\cite[Section~9.1.1]{HoltHandbookCGT},
we find that $\Gamma$ has $34$ normal subgroups $N$ such that
$\Gamma/N\cong\SL(3,5)$. Applying our algorithm, we identify $80$ values of 
$k$ in the range $1,\ldots,100$, such that $5\not\in\Pi(\rho_k(\Gamma))$.
For these $k$, the kernels of the induced surjections
 $\Gamma \rightarrow \rho_k(\Gamma) \rightarrow \SL(3, 5)$ 
expose just $4$ of
the $34$ normal subgroups. This prompts us to conjecture that the
$\rho_k$ will not expose all $\SL(n,p)$ quotients of $\Gamma$. 
 
\subsubsection*{Acknowledgments}

We thank Mathematisches Forschungsinstitut Oberwolfach 
and the International Centre for Mathematical Sciences, Edinburgh, 
for hosting our visits in 2017 under their Research-in-Pairs and 
Research in Groups programmes, respectively.
This work was additionally supported by a 
Marie Sk\l odowska-Curie Individual 
Fellowship grant (Horizon 2020, EU Framework Programme for 
Research and Innovation), and Simons Foundation Collaboration 
Grant 244502.

\begin{landscape}
\begin{table}

\begin{tabular}{l|r|r|l}
Group
&Level
&Index
&Quotient
\\
\hline
$H_1(1)$
&$2^{2}3
{\cdot}5^{2}19$
&$
\vphantom{\displaystyle\int_0^0}%
2^{13}3^{3}5^{5}13
{\cdot}19^{3}31
{\cdot}127$
&$\ppqf{\ }{2^2}.\ppqf{5^6}{5^2}.\ppqf{3}{19}.\ppqf{2.2}{5,19}.\ppqf{3^2.2}{3}.\
\ppqf{2^2.3}{3,5}.\ppqf{A_6}{19}.\ppqf{L_3(2)}{2}$
\\
$H_1(2)$
&$2^{3}$
&$2^{7}7$
&$\ppqf{2^5}{2^3}.\ppqf{2^5}{2^2}.\ppqf{2^2.3}{2}$
\\
$H_1(5)$
&$2^{3}7
{\cdot}19
{\cdot}31$
&$2^{23}3^{6}5^{2}7^{2}19^{2}31^{3}127
{\cdot}331$
&$\ppqf{\ }{2^3}.\ppqf{\ }{2^2}.\ppqf{3}{31}.\ppqf{19^2.2.2.2}{19}.\ppqf{3}{7,19,\
31}.\ppqf{A_6}{31}.\ppqf{L_2(7)}{7}.\ppqf{L_3(2)}{2}$
\\
$H_1(9)$
&$2^{2}67$
&$2^{9}3^{2}7^{2}11^{2}17
{\cdot}31
{\cdot}67$
&$\ppqf{\ }{2^2}.\ppqf{67^2.2.2^2.3}{67}.\ppqf{L_3(2)}{2}$
\\
$H_1(10)$
&$2^{3}3
{\cdot}7$
&$2^{16}3^{4}7^{2}13
{\cdot}19$
&$\ppqf{2^5}{2^3}.\ppqf{2^4}{2^2}.\ppqf{7.7}{7}.\ppqf{3.3.2}{3,7}.\ppqf{2^2.3}{\
2,3,7}$
\\
$H_1(12)$
&$2^{3}7
{\cdot}31$
&$2^{16}3^{5}5^{2}7^{3}19
{\cdot}31
{\cdot}331$
&$\ppqf{2^5}{2^3}.\ppqf{2^4}{2^2}.\ppqf{31.31.2}{31}.\ppqf{2^2}{2,31}.\ppqf{3}{\
2,7,31}.\ppqf{L_2(7)}{7}$
\\
$H_1(50)$
&$2^{3}601$
&$2^{14}3^{3}5^{4}7^{2}13
{\cdot}43
{\cdot}601
{\cdot}9277$
&$\ppqf{2^5}{2^3}.\ppqf{2^4}{2^2}.\ppqf{601.601.2}{601}.\ppqf{2^2.3}{2,601}$
\\
$H_1(100)$
&$2^{3}3
{\cdot}19
{\cdot}43$
&$2^{19}3^{9}5
{\cdot}7^{3}11
{\cdot}13
{\cdot}19
{\cdot}43
{\cdot}127
{\cdot}631$
&$\ppqf{2^5}{2^3}.\ppqf{2^4}{2^2}.\ppqf{43.43}{43}.\ppqf{19.19}{19,43}.%
\ppqf{3.3.2}{3,19,43}.\ppqf{2^2.3}{2,3,19,43}$
\\
$H_2(2)$
&$2^{2}13$
&$2^{7}3^{2}7
{\cdot}61$
&$\ppqf{2^2}{2^2}.\ppqf{13.13.2.2}{13}.\ppqf{2^2.3}{2,13}.\ppqf{L_2(13)}{13}.%
\ppqf{2}{2,13}$
\\
$H_2(10)$
&$2^{2}5
{\cdot}109$
&$2^{10}3^{4}5^{2}7^{2}31
{\cdot}571$
&$\ppqf{2^2}{2^2}.\ppqf{109.109.2.2}{109}.\ppqf{2^2}{2,109}.\ppqf{3}{2}.%
\ppqf{L_2(109)}{109}.\ppqf{A_5}{5}.\ppqf{2}{2,5,109}$
\\
$H_2(12)$
&$2^{8}3^{4}17$
&$2^{38}3^{15}7
{\cdot}13
{\cdot}307$
&
$\ppqf{2^6}{2^8}.\ppqf{2^6}{2^7}.\ppqf{2^5}{2^6}.\ppqf{2^3}{2^5}.%
\ppqf{2^3}{2^4}.\ppqf{3^4}{3^4}.\ppqf{2^2}{2^3}.\ppqf{3^3}{3^3}.\ppqf{\ }{2^2}.%
\ppqf{3^3}{3^2}.\ppqf{17^2.2}{17}.\ppqf{3.3.2}{3,17}.
\ppqf{2^2}{2,3,17}.\ppqf{3}{2,17}.%
\ppqf{2}{2,3,17}.\ppqf{L_2(17)}{17}$
\\
$H_2(50)$
&$2^{2}5^{2}13
{\cdot}193$
&$2^{16}3^{4}5^{7}7^{2}31
{\cdot}61
{\cdot}1783$
&
$\ppqf{2^2}{2^2}.\ppqf{5^3}{5^2}.\ppqf{193.193.2}{193}.%
\ppqf{13.13.2.2}{13,193}.\ppqf{2^2}{2,193}.\ppqf{3}{2}.
\ppqf{L_2(193)}{193}.\ppqf{L_2(13)}{13}.\ppqf{A_5}{5}.\ppqf{2}{2,5,13,193}$
\\
$H_3(0)$
&$2^{3}7^{3}19^{2}$
&
$2^{52}3^{14}5^{5}7^{32}19^{21}31 {\cdot}127%
{\cdot}151
{\cdot}181
{\cdot}911
{\cdot}2801$
&$\ppqf{2^8}{2^3}.\ppqf{7^18}{7^3}.\ppqf{2^6}{2^2}.\ppqf{7^8}{7^2}.%
\ppqf{19^{12}}{19^2}.\ppqf{19^2.2.2.2}{19}.\ppqf{2^6.3}{2,19}.\ppqf{L_3(2)}{2,7}$
\\
$H_3(1)$
&$2^{6}67^{2}$
&
$2^{87}3^{6}5^{2}7^{3}11^{4}17^{2}31^{2}67^{20}
\cdot449
{\cdot}761
{\cdot}26881$
&$\ppqf{2^{21}}{2^6}.\ppqf{2^{11}}{2^5}.\ppqf{2^8}{2^4}.\ppqf{2^5}{2^3}.%
\ppqf{2^4}{2^2}.\ppqf{67^{12}}{67^2}.\ppqf{67.67.2}{67}.\ppqf{2^2.3}{2,67}$
\\
$H_3(2)$
&$2^{3}13
{\cdot}211^{2}$
&\begin{minipage}[t]{6cm}
$2^{33}3^{9}5^{6}7^{5}13^{5}31^{2}37
{\cdot}53^{2}61
{\cdot}113
{\cdot}197$
$\hphantom{XX}
{\cdot}211^{20}
\cdot1361
{\cdot}30941
{\cdot}292661$
\end{minipage}
&
$\ppqf{2^{16}}{2^3}.\ppqf{2^{14}}{2^2}.\ppqf{211^{12}}{211^2}.%
\ppqf{211^2.2.2.2}{211}.\ppqf{13.13.13.13}{13,211}.
\ppqf{2^6.3}{2,13,211}.\ppqf{L_2(169)}{13}.%
\ppqf{L_3(2)}{2}$
\\
$H_3(3)$
&$2^{6}7
{\cdot}11^{2}41^{2}$
&\begin{minipage}[t]{6cm}
$
\vphantom{\displaystyle\int_0^0}%
2^{106}3^{8}5^{13}7^{8}11^{20}19^{2}
\cdot 29^{2}31
{\cdot}41^{20}61
$
$\hphantom{XX}
{\cdot}1723
{\cdot}2801
{\cdot}3221
{\cdot}579281$
\end{minipage}
&\begin{minipage}[t]{6cm}
$\ppqf{2^{21}}{2^6}.\ppqf{2^{11}}{2^5}.\ppqf{2^8}{2^4}.\ppqf{2^5}{2^3}.%
\ppqf{2^4}{2^2}.\ppqf{11^{12}}{11^2}.\ppqf{41^{12}}{41^2}.\ppqf{41.41}{41}.%
\ppqf{11.11.2}{11,41}.$
$\hphantom{XX}\ppqf{7.7.7.7.2}{7,11,41}.\ppqf{2^2.3}{2,7,11,41}.%
\ppqf{L_2(7).L_2(7)}{7}$
\end{minipage}
\\
\end{tabular}

\medskip

\caption{}
\label{Table3}
\end{table}
\end{landscape}

\bibliographystyle{amsplain}

\begin{thebibliography}{10}

\bibitem{Aoun}
R.~Aoun, \emph{Random subgroups of linear groups are free}, 
Duke Math. J.  \textbf{160} (2011), no.~1, 117--173.

\bibitem{Aschbacher84}
M.~Aschbacher, \emph{On the maximal subgroups of the finite classical groups}, 
Invent. Math. \textbf{76} (1984), 469--514.

\bibitem{Bray}
J.~N. Bray, D.~F. Holt, and C.~M. Roney-Dougal, 
\textit{The maximal subgroups of the low-dimensional finite classical groups}.
London Math. Soc. Lecture Note Ser. \textbf{407}, Cambridge Univ. 
Press, Cambridge, 2013.

\bibitem{Breuillard}
E.~Breuillard, \textit{Approximate subgroups and super-strong
approximation}, Groups St Andrews 2013, 1--50,
London Math. Soc. Lecture Note Ser. \textbf{422}, Cambridge Univ. Press, Cambridge, 2015.

\bibitem{Atlas}
J. H. Conway, S. P. Norton, R. A. Wilson, R. T. Curtis, and R. A. Parker, 
\textit{Atlas of finite groups: maximal subgroups 
and ordinary characters for simple groups}, Oxford, 1986.

\bibitem{Arithm}
A.~S. Detinko, D.~L. Flannery, and A.~ Hulpke, 
\textit{Algorithms for
arithmetic groups with the congruence subgroup property}, 
J. Algebra \textbf{421} (2015), 234--259.

\bibitem{Density}
A.~S. Detinko, D.~L. Flannery, and A.~ Hulpke, 
\textit{Zariski density and computing in arithmetic groups}, 
Math. Comp. \textbf{87} (2018), no.~310, 967--986. 

\bibitem{GAPDoc}
A.~S. Detinko, D.~L. Flannery, and A.~ Hulpke,
\emph{GAP functionality for Zariski dense groups}, 
Oberwolfach Preprints OWP 2017-22.

\bibitem{SAT_General}
A.~S. Detinko, D.~L. Flannery, and A.~ Hulpke, 
\textit{Strong approximation and algorithms for computing with dense subgroups}, 
preprint, 2018.
 
\bibitem{Recog}
A.~S. Detinko, D.~L. Flannery, and E.~A. O'Brien,
\textit{Recognizing finite matrix groups over infinite
fields}, J. Symbolic Comput. \textbf{50} (2013), 100--109.

\bibitem{Tits}
A.~S. Detinko, D.~L. Flannery, and E.~A. O'Brien,
\textit{Algorithms for the Tits alternative and related problems},
J. Algebra \textbf{344} (2011), 397--406

\bibitem{Feitpreprint}
W.~Feit, \textit{The orders of finite linear groups},
preprint, 1995.

\bibitem{FuchsRivin}
E.~Fuchs and I.~Rivin, \emph{Generic thinness in finitely generated subgroups
  of ${\rm SL}_n(\mathbb{Z})$}, Int. Math. Res. Not. IMRN (2017), no.~17,
  5385--5414.

\bibitem{GAP}
The~GAP Group, \emph{{GAP} -- {G}roups, {A}lgorithms, {P}rogramming}, 
\url{http://www.gapsystem.org}.

\bibitem{HoltHandbookCGT}
D.~F. Holt, B.~Eick, and E.~A. O'Brien, \emph{Handbook of computational group
  theory}, Discrete Mathematics and its Applications, Chapman \&
  Hall/CRC, Boca Raton, FL, 2005.

\bibitem{SteinbergRep}
J. E. Humphreys, \emph{The Steinberg representation}, Bull. Amer. Math. Soc.
(N.S.) \textbf{16}, (1987), no.~2, 247--263.

\bibitem{James83}
G. D. James, \textit{On the minimal dimensions of irreducible representations
of symmetric groups}. Math. Proc. Cambridge Philos. Soc. \textbf{94} (1983), 
no.~3, 417--424.

\bibitem{KleidmanLiebeck}
P. Kleidman and M. Liebeck,
\textit{The subgroup structure of the finite classical groups}.
London Math. Soc. Lecture Note Ser. \textbf{129}, 
Cambridge Univ. Press, Cambridge, 1990. 

\bibitem{KleshchevTiep12}
A.~S. Kleshchev and P.~H. Tiep,
\textit{Small-dimensional projective representations of symmetric and
alternating groups}. Algebra Number Theory \textbf{6} (2012), 1773--1816.

\bibitem{LongReidI}
D.~D.~Long, A.~W.~Reid, \textit{Small subgroups of $\SL(3, \mathbb{Z})$}, 
Exp. Math.  \textbf{20} (2011),
no.~4, 412--425.

\bibitem{LongReidII}
D.~D.~Long, A.~W.~Reid, and M. Thistlethwaite, \textit{Zariski dense 
surface subgroups in $\SL(3, \mathbb{Z})$}, 
Geom. Topol.~\textbf{15} (2011) 1-9.

\bibitem{LongNew} 
D.~D.~Long and M.~Thistlethwaite,  
\textit{Zariski dense surface subgroups in $\SL(4, \mathbb{Z})$}, 
Exp. Math. \textbf{27} (2018), no.~1, 82--92.

\bibitem{LubotzkySegal}
A.~Lubotzky and D.~Segal, \textit{Subgroup growth}, Birkh\"{a}user, Basel, 2003.

\bibitem{Lubotzky97}
A.~Lubotzky, \textit{One for almost all: generation of $\SL(n,p)$ by subsets 
of $\SL(n,\mathbb{Z})$}, Contemp. Math. \textbf{243}, 125--128, 1999.

\bibitem{MalleZaleskii}
G. Malle and A.~E. Zalesskii, 
\textit{Prime power degree representations of quasi-simple groups},
Arch.~Math. (Basel) \textbf{77} (1981), 
461--468.

\bibitem{gaprecog}
M.~Neunh\"{o}ffer, \'{A}.~Seress, et al., The {\sf GAP} package {\tt recog},
\emph{a collection of group recognition methods},
\url{http://gap-packages.github.io/recog/}.

\bibitem{Rivin}
I. Rivin,
\textit{Zariski density and genericity}, 
Int. Math. Res. Not. IMRN 2010, no. 19, 3649--3657.

\bibitem{Rivin2}
I.~Rivin, 
\textit{Generic phenomena in groups: some answers and many questions}. 
Thin groups and superstrong approximation, 
299--324, Math. Sci. Res. Inst. Publ. {\bf 61}, Cambridge Univ. Press, Cambridge, 2014. 

\bibitem{Sarnak} P.~Sarnak, \textit{Notes on thin matrix groups},
Thin groups and superstrong approximation, 343--362, Math. Sci.
Res. Inst. Publ. \textbf{61}, Cambridge Univ. Press, Cambridge, 2014.

\bibitem{Steinberg85}
R.~Steinberg, \emph{Some consequences of the elementary relations in {${\rm
  SL}_n$}}, Finite groups---coming of age ({M}ontreal, {Q}ue., 1982), Contemp.
  Math., vol.~45, Amer. Math. Soc., Providence, RI, 1985, pp.~335--350.

\bibitem{Supr} D. A. Suprunenko, \textit{Matrix groups}, 
Transl. Math. Monogr., vol. 45, American Mathematical
Society, Providence, RI, 1976.

\bibitem{Wehrfritz}
B.~A.~F. Wehrfritz, \emph{Infinite linear groups},
Springer-Verlag, New York, 1973.

\bibitem{Weigel}
T.~S.~Weigel, \textit{On the profinite completion of arithmetic groups
of split type},
Lois d'alg\`ebres et vari\'et\'es alg\'ebriques ({C}olmar,
1991), 79--101,  Travaux en Cours, \textbf{50},
Hermann, Paris, 1996.

\bibitem{ZalesSeregkin}
A.~E.~Zalesskii, V.~N.~Sere\v{z}kin, \textit{Linear groups generated
by transvections}, Izv. Akad. Nauk SSSR Ser. Mat. \textbf{40} (1976), no.~1, 
26--49 (Russian).

\bibitem{ZalessII}
A. E. Zalesskii, \textit{Linear groups}, 
Algebra, IV, Encyclopaedia Math. Sci., \textbf{37}, pp. 97--196,
Springer, Berlin, 1993.

\end{thebibliography}

\end{document}